\documentclass[10pt,letterpaper, boxed, cm]{amsart}

\usepackage[margin=1in]{geometry}
\usepackage{amsmath,amsthm,amsfonts,amssymb,amscd}
\usepackage{lastpage}
\usepackage{fancyhdr}
\usepackage{mathrsfs}
\usepackage{xcolor}
\usepackage{graphicx}
\usepackage{listings}
\usepackage{hyperref}
\usepackage{enumitem}
\usepackage{relsize}
\usepackage{tikz-cd}
\usepackage{mdwlist}
\usepackage{multicol}
\usepackage{stmaryrd}
\usepackage{float}
\usepackage{adjustbox}
\usepackage{tikz}
\usetikzlibrary{matrix, calc, arrows}

\hypersetup{
    colorlinks=true, 
    linktoc=all,     
    linkcolor=blue,  
}

\setlist[enumerate]{label=(\alph*)}

\usetikzlibrary{graphs,decorations.pathmorphing,decorations.markings}
\tikzcdset{scale cd/.style={every label/.append style={scale=#1},
        cells={nodes={scale=#1}}}}

\newcommand{\Z}{\mathbb{Z}}

\newcommand{\F}{\mathbb{F}}

\newcommand{\Hom}{\operatorname{Hom}}

\newcommand{\tr}{\operatorname{tr}}

\newcommand{\inv}{^{-1}}

\newcommand{\End}{\operatorname{End}}

\newcommand{\Res}{\operatorname{Res}}
\newcommand{\Ind}{\operatorname{Ind}}

\newcommand{\id}{\operatorname{id}}

\newcommand{\stab}{\operatorname{stab}}

\newcommand{\Syl}{\operatorname{Syl}}

\newcommand{\triv}{\mathbf{triv}}

\newcommand{\calC}{\mathcal{C}}

\newcommand{\calL}{\mathcal{L}}

\newcommand{\calO}{\mathcal{O}}

\newcommand{\Gal}{\operatorname{Gal}}
\newcommand{\Stilde}{\tilde{S}}
\newcommand{\etilde}{\tilde{e}}
\newcommand{\btilde}{\tilde{b}}
\newcommand{\Vtilde}{\tilde{V}}

\newcommand{\ftilde}{\tilde{f}}
\newcommand{\ctilde}{\tilde{c}}
\newcommand{\Ttilde}{\tilde{T}}

\newcommand{\catmod}{\mathbf{mod}}

\newtheorem{theorem}{Theorem}
\newtheorem{lemma}[theorem]{Lemma}
\newtheorem{prop}[theorem]{Proposition}
\newtheorem{corollary}[theorem]{Corollary}

\newtheorem*{theorem*}{Theorem}

\theoremstyle{remark}

\theoremstyle{definition}

\begin{document}
    \title{Galois descent of splendid Rickard equivalences for blocks of $p$-nilpotent groups}
    \author{Sam K. Miller}
    \address{Department of Mathematics, University of California Santa Cruz, Santa Cruz, CA 95064} 
    \email{sakmille@ucsc.edu} 
    \begin{abstract}
        We strengthen results of Boltje and Yilmaz regarding Galois descent of equivalences of blocks of $p$-nilpotent groups, and a result of Kessar and Linckelmann regarding Galois descent of splendid Rickard equivalences for blocks with compatible Galois stabilizers. A more general descent criteria for chain complexes is proven along the way, which requires the adaptation of a theorem of Reiner for chain complexes. This verifies Kessar and Linckelmann's refinement of Brou\'{e}'s abelian defect group conjecture for blocks of $p$-nilpotent groups with abelian Sylow $p$-subgroup.   
    \end{abstract}
    
    \maketitle
    
    Recently, there has been new interest in verifying Brou\'{e}'s abelian defect group conjecture over \textit{arbitrary fields}, in particular over $\F_p$, rather than the standard assumption that one works over a splitting field. This interest stems from a refined version of the conjecture proposed by Kessar and Linckelmann in \cite{KL18}, which predicts that for \textit{any} complete discrete valuation ring $\calO$ and any block of a finite group $G$ over $\calO$ with abelian defect group, there is a splendid Rickard equivalence between the block algebra and its Brauer correspondent. Given a $p$-modular system $(K, \calO, k)$, one may use unique lifting of splendid Rickard equivalences from $k$ to $\calO$ to reformulate the strengthened conjecture over fields of positive characteristic. Brou\'{e}'s conjecture is known over $\F_p$ in some instances, such as in the case of symmetric groups \cite{CR08}. In most cases however, Brou\'{e}'s conjecture has only been verified under the more common assumption that the field $k$ is a splitting field for the group $G$.
    
    Recent papers such as \cite{H23}, \cite{HLZ23}, and \cite{KL18} have verified Kessar and Linckelmann's strengthened conjecture for cases in which Brou\'{e}'s conjecture was already known to hold. One common approach to verify the conjecture over $\F_p$ is to take a previously known construction of a splendid Rickard equivalence and modify it such that it descends to the base field $\F_p$ from a splitting field extension $k/ \F_p$. Brou\'{e}'s conjecture is known to hold for $p$-nilpotent groups, that is, finite groups whose largest $p'$-normal subgroup $N$ is a compliment to a Sylow $p$-subgroup of $G$ (and more generally holds for $p$-solvable groups due to \cite{HL00}). In \cite{BY22}, Boltje and Yilmaz determined two cases for when equivalences of block algebras over $\F_p$ of $p$-nilpotent groups exist by modifying previously known techniques. However, in one of the two cases, \cite[Theorem B]{BY22}, only a weaker form of equivalence known as a $p$-permutation equivalence (as introduced by Boltje and Xu in \cite{BX07}) was shown. In this paper, we improve their result by demonstrating the existence of splendid Rickard equivalences in the same setting.
    
    \begin{theorem}{(\cite[Theorem B]{BY22} for splendid Rickard equivalences)}\label{thm:thmbgen}
        Let $G$ be a $p$-nilpotent group with abelian Sylow $p$-subgroup and let $\tilde{b}$ be a block idempotent of $\F_pG$. Then there exists a splendid Rickard equivalence between $\F_pG\tilde{b}$ and its Brauer correspondent block algebra. In particular, Kessar and Linckelmann's strengthened abelian defect group conjecture holds for blocks of $p$-nilpotent groups with abelian Sylow $p$-subgroups.
    \end{theorem}
    
    In fact, Theorem B follows from \cite[Corollary 5.15]{BY22}, which we also prove an analogue of. The key tool for proving these theorems is a splendid Rickard equivalence-theoretic analogue of \cite[Theorem D]{BY22}. This theorem also strengthens a descent result of Kessar and Linckelmann, \cite[Theorem 6.5(b)]{KL18}, by allowing for descent from larger finite fields, rather than only the field of realization of a block. 
    
    \begin{theorem}{(\cite[Theorem D]{BY22} for splendid Rickard equivalences)}\label{thm:thmdgen}
        Let $b$ and $c$ be block idempotents of $kG$ and $kH$. Let $\tilde{b}$ and $\tilde{c}$ denote the corresponding block idempotents of $\F_pG$ and $\F_pH$ associated to $b$ and $c$ respectively, i.e. the unique block idempotents of $\F_pG$ and $\F_pH$ for which $b\tilde{b}\neq 0$ and $c\tilde{c} \neq 0$. Set $\Gamma := \Gal(k/\F_p)$. Then explicitly, \[\tilde{b} = \tr_\Gamma(b) = \sum_{\sigma \in \Gamma/\stab_\Gamma(b)} {}^\sigma b, \text{ and } \tilde{c} = \tr_\Gamma(c) = \sum_{\sigma \in \Gamma/\stab_\Gamma(c)} {}^\sigma c.\]
        
        Let $X$ be a splendid Rickard equivalence between $kGb$ and $kHc$. Suppose we have $\stab_\Gamma(X) = \stab_\Gamma(b) = \stab_\Gamma(c)$. Then $\F_pG\tilde{b}$ and $\F_pH\tilde{c}$ are splendidly equivalent.
    \end{theorem}
    
    \textbf{Notation and conventions.}
    For this paper, we follow the notation and setup of \cite{KL18}. Let $(K, \calO, k) \subseteq (K', \calO', k')$ be an extension of $p$-modular systems, i.e. a $p$-modular system such that $\calO \subseteq \calO'$ with $J(\calO) \subseteq J(\calO')$. Furthermore, assume $k$ and $k'$ are finite, and that $\calO$ and $\calO'$ are absolutely unramified unless otherwise specified. In other words, we have $\calO \cong W(k)$ and $\calO' \cong W(k')$, where $W(k)$ denotes the ring of Witt vectors over $k$. Equivalently, we have $J(\calO) = p\calO$ and $J(\calO') = p\calO'$. Set $d = [k' : k]$. Then $\calO'$ is free of rank $d$ as an $\calO$-module. Let $\sigma: k' \to k'$ be a generator of $\Gamma := \Gal(k'/k)$ (such as the Frobenius endomorphism) and denote by the same letter $\sigma: \calO' \to \calO'$ the unique ring homomorphism of $\calO'$ lifting $\sigma$. This lift is unique since $\calO$ and $\calO'$ are absolutely unramified. 
    
    We use the following notation for Galois twists for extensions of commutative rings $\calO \subseteq \calO'$. Given a finitely generated $\calO$-algebra $A$, a module $U$ over the $\calO'$-algebra $A' = \calO' \otimes_\calO A$, and a ring automorphism $\sigma$ of $\calO'$ which restricts to the identity map on $\calO$, we denote by ${}^\sigma U$ the $A'$-module which is equal to $U$ as a module over the subalgebra $1 \otimes A$ of $A'$, such that $\lambda \otimes a$ acts on $U$ as $\sigma\inv(\lambda) \otimes a$ for all $a \in A$, $\lambda \in \calO'$. The Galois twist induces an $\calO$-linear (but not in general $\calO'$-linear) self-equivalence on ${}_{A'}\catmod$. The same notation extends to chain complexes in the obvious way. Note the Krull-Schmidt theorem for chain complexes holds over $A$, see \cite[Theorem 4.6.11]{L181}.
    
    We obtain functors \[-_{\calO}: {}_{A'}\catmod \to {}_{A}\catmod \text{ and } \calO' \otimes_\calO -: {}_{A}\catmod \to {}_{A'}\catmod,\] restriction and extension of scalars, respectively. These are both $\calO$-linear exact functors. Moreover, $\calO'\otimes_\calO -$ is left adjoint to $ -_\calO$. We also have restriction and extension of scalars for $k$-algebras, and the same adjunction holds. Moreover, these functors extend to functors over chain complex categories, and the same adjunctions hold. 
    
    Let $\tau: {}_{A'}\catmod \to {}_{A'}\catmod$ be the functor sending an $A'$-module $U$ to the $A'$-module \[\tau(U) := \bigoplus_{i = 0}^{d-1} {}^{\sigma^i}U,\] and a morphism $f$ to \[\tau(f) := (f,\dots, f).\]
    $\tau$ is an exact functor of $\calO$-linear categories, where we regard ${}_{A'}\catmod$ as an $\calO$-linear category by restriction of scalars. 
    
    \begin{prop}{\cite[Proposition 6.3]{KL18}}
        With the notation and assumptions above, the functors $\calO' \otimes_\calO (-)_\calO$ and $\tau$ are naturally isomorphic. That is, for any $A'$-module $U$, we have a natural isomorphism \[\calO' \otimes_\calO U_\calO \cong \tau(U).\]
    \end{prop}
    
    It easily follows that the above statement holds as well for $k', k$ replacing $\calO', \calO$. We next require an adaptation of a theorem of Reiner for chain complexes. This next lemma is an extension of \cite[Proposition 15]{R66} for chain complexes.
    
    \begin{lemma}
        Let $C \in Ch^b({}_A\catmod)$, and set $E(C) := \End_A(C)$ and $\tilde{E}(C) := E(C)/J(E(C))$. Each decomposition of $C' := \calO' \otimes_\calO C$ into a direct sum of indecomposable subcomplexes \[C' = \bigoplus_{i=1}^l D_j\] induces a corresponding decomposition of $k'\otimes_k \tilde{E}(C)$ into indecomposable left ideals, \[k' \otimes_k \tilde{E}(C) = \bigoplus_{i=1}^l L_j,\] where $D_j$ and $L_j$ correspond via the idempotent $\tau_j \in E(C')$ for which $\tau_j(C') = D_j$. Moreover, for all $i, j$ we have $D_i \cong D_j$ as $A$-complexes if and only if $L_i \cong L_j$ as left $k'\otimes_k \tilde{E}(C)$-modules.
    \end{lemma}
    \begin{proof}
        First, note that for any chain complexes $C,D \in Ch^b({}_{A}\catmod)$, we have \[\Hom_{A'}(C', D') \cong \Hom_A(C, D'_\calO) \cong \calO' \otimes_\calO \Hom_A(C,D),\] where the first isomorphism arises from adjunction, and the second isomorphism is given by
        \begin{align*}
            \calO' \otimes_\calO \Hom_A(C,D) & \to \Hom_A(C, D'_\calO)\\
            a \otimes f &\mapsto \big(m_i \mapsto a \otimes f_i(m_i) \big)_{i \in \Z}
        \end{align*}
        where $a\in \calO'$ and $f = \{f_i:C_i \to D_i\}_{i\in \Z}$ is a chain complex homomorphism $C \to D$ (the author thanks D. Benson for his suggestion \cite{471654}). One may verify this is a well-defined homomorphism for chain complexes. Moreover, the homomorphism is surjective, as $\Hom_A(C,D'_\calO)$ is spanned by homomorphisms of the form $m \mapsto a \otimes f_i(m)$ in degree $i$, where $a \in \calO'$ and $f = \{f_i:C_i \to D_i\}_{i\in \Z}$ is a chain complex homomorphism $C \to D$. This follows from the identification $D'_\calO \cong D^{\oplus d}$ of $A$-chain complexes, which implies $\Hom_A(C,D'_\calO) \cong \Hom_A(C,D)^{\oplus d}$. The surjection is therefore injective as well, by rank counting. Moreover if $C = D$, the composite of the two isomorphisms is an $\calO'$-algebra isomorphism, as in this case the composite is as follows: 
        \begin{align*}
            \calO'\otimes_\calO \End_A(C) &\to \End_{A'}(C')\\
            a\otimes f &\mapsto m_a \otimes f
        \end{align*}
        where $m_a$ is the multiplication by $a$ map. Therefore, we have an isomorphism $E(C') = \End_{A'}(C') \cong \calO' \otimes_\calO E(C)$ of $\calO'$-algebras. Now, choose \[I = p\calO' \otimes_\calO E(C) + \calO' \otimes_\calO J(E(C)),\] then $I$ is a two-sided ideal of $\calO'\otimes_\calO E(C)$ contained in $J(\calO'\otimes_\calO E(C))$. We may also regard $I$ as a two-sided ideal of $E(C')$ via the isomorphism $E(C') \cong \calO'\otimes_\calO E(C).$ Since $I$ is contained in $J(E(C'))$, $E(C')$ decomposes into left ideals in the same way as the factor ring $E(C')/I$. However, \[E(C')/I \cong (\calO'\otimes_\calO E(C))/I \cong k' \otimes_k \tilde{E}(C).\] 
        
        Now, \cite[Corollary 4.6.12]{L181} asserts that a decomposition $C' = \bigoplus_{i=1}^l D_j$ into indecomposable summands corresponds bijectively to a decomposition of $\id \in E(C')$ into primitive idempotents $\tau_i$ with $\tau_i(C') = D_i$ and $\tau_i$ is $E(C')$-conjugate to $\tau_j$ if and only if $D_i \cong D_j$. Therefore, we obtain a corresponding decomposition of $E(C') = \bigoplus_{i=1}^l K_j$ into indecomposable left ideals, with $K_i =  E(C') \tau_i$, $E(K_i) \cong \tau_i E(C') \tau_i$, and moreover, $D_i \cong D_j$ if and only if $K_i \cong K_j$ as $E(C')$-modules. On the other hand, since $I \subseteq J(E(C'))$, the decomposition $E(C') = \bigoplus_{i=1}^l K_j$ also corresponds bijectively to a decomposition $E(C')/I = \bigoplus_{i=1}^l L_j$ into indecomposable left ideals, with $L_i \cong L_j$ if and only if $K_i \cong K_j$. The result follows. 
        
    \end{proof}
    
    The following proposition is an extension of \cite[Theorem 3]{R66} for chain complexes.
    
    \begin{prop}\label{thm:reinergen}
        For each indecomposable chain complex of finitely generated $A$-modules $C$, the chain complex of $A'$-modules $C' := \calO'\otimes_\calO C$ is a direct sum of nonisomorphic indecomposable subcomplexes.
    \end{prop}
    
    \begin{proof}
        We use the notation from the previous lemma. Note $\tilde{E}(C)$ is semisimple, but since $C$ is indecomposable, $\tilde{E}(C)$ is a division algebra over $k$. However, $k$ is finite; hence $\tilde{E}(C)$ is as well, and thus a field by Wedderburn's theorem. Then $k' \otimes_k \tilde{E}(C)$ is a semisimple commutative algebra by \cite[Theorem 7.8]{CR81}. Therefore, it is a direct sum of fields, none of which are isomorphic as $k'\otimes_k \tilde{E}(C)$-modules. The result now follows from the previous lemma. 
    \end{proof}

    The following theorem demonstrates the necessary and sufficient condition for a chain complex of $A'$-modules to descend to a chain complex of $A$-modules is Galois stability. 
    
    \begin{theorem}\label{thm:galoisstablelift}
        \begin{enumerate}
            \item Suppose $C \in Ch^b({}_{A'}\catmod)$ is indecomposable and satisfies ${}^\sigma C \cong C$ for all $\sigma \in\Gal(k'/k)$, where we regard $\sigma$ as the unique ring homomorphism of $\calO'$ lifting $\sigma \in\Gal(k'/k)$. Then there exists a chain complex $\tilde{C} \in Ch^b({}_{A}\catmod)$ such that $\calO' \otimes_\calO \tilde{C} \cong C$. Moreover, $\tilde{C}$ is unique up to isomorphism.
            \item Conversely, let $\tilde{C} \in Ch^b({}_{A}\catmod)$ and define $C := \calO' \otimes_\calO \tilde{C}$. Then $C$ satisfies ${}^\sigma C \cong C$ for all $\sigma \in\Gal(k'/k)$.
        \end{enumerate}  
    \end{theorem}
    \begin{proof}
        For (a), let $d = [k':k]$. By \cite[Proposition 6.3]{KL18}, there is a natural isomorphism $\calO' \otimes_\calO C_\calO \cong C^{\oplus d}$. We claim $\calC_\calO$ has exactly $d$ indecomposable summands. Indeed, suppose for contradiction that $C_\calO$ contains less than $[k':k]$ indecomposable summands. Since extension of scalars is an exact functor, there exists a summand $D$ of $C_\calO$ for which $\calO'\otimes_\calO D \cong C^{\oplus i}$ for some $1 < i \leq d$ by the Krull-Schmidt theorem. However, this contradicts Proposition \ref{thm:reinergen} which asserts that $\calO'\otimes_\calO D$ decomposes into a direct sum of nonisomorphic indecomposable subcomplexes. On the other hand, $C_\calO$ cannot contain more than $d$ indecomposable summands since extension of scalars is exact. Thus, $C_\calO$ contains exactly $d$ indecomposable summands, $C_1, \dots, C_{d} \in Ch^b({}_{A}\catmod)$. We have \[\bigoplus_{i=1}^d(\calO' \otimes_\calO C_i)\cong \calO' \otimes_\calO (C_1 \oplus \cdots \oplus C_d) = \calO' \otimes_\calO C_\calO = \tau(C) \cong C^{\oplus d},\] and by the Krull-Schmidt theorem, choosing $\tilde{C} := C_i$ for any $i \in \{1,\dots, d\}$ demonstrates the first statement in (a). 
        
        For uniqueness in (a), suppose that $\tilde{C}_1, \tilde{C}_2 \in Ch^b({}_{A}\catmod)$ satisfy $C \cong \calO' \otimes_\calO \tilde{C}_1 \cong \calO' \otimes_\calO \tilde{C}_2$. Since $\calO'$ is a free $\calO$-module of rank $d$, we may take an $\calO$-basis of $\calO'$, $\{a_1,\dots, a_d\}$. Then, $(\calO' \otimes_\calO \tilde{C}_1)_\calO \cong \tilde{C}_1^{\oplus d}$, since $a_i \otimes_\calO \tilde{C}_1$ is an $A$-direct summand of $\calO'\otimes_\calO \tilde{C}_1$ for $i \in \{1, \dots, d\}$. Similarly, $(\calO' \otimes_\calO \tilde{C}_2)_\calO \cong \tilde{C}_2^{\oplus d}$, so we have $\tilde{C}_1^{\oplus d}\cong \tilde{C}_2^{\oplus d}$ and by the Krull-Schmidt theorem, $\tilde{C}_1 \cong \tilde{C}_2$, as desired. 
        
        For (b), given any $A$-module $M$ and $\sigma \in \Gal(k'/k)$, we have an isomorphism of $A'$-modules $\calO' \otimes_\calO M \cong {}^\sigma(\calO' \otimes_\calO M)$, natural in $M$, given by $a \otimes m \mapsto \sigma\inv(a) \otimes m$. The result follows by applying this isomorphism in each component of $C$.  
    \end{proof}
    
    As a result, Galois stability allows us to determine when a splendid Rickard equivalence descends to the field which realizes two splendidly Rickard equivalent blocks. 

    \begin{lemma}\label{lem:faithfulact}
        Let $k'/k$ be a finite field extension, let $G$ be a finite group, and let $b$ be a block idempotent of $kG$. Then $\Gal(k'/k)$ acts transitively on the set of block idempotents $b'$ of $k'G$ for which $b'b \neq 0$. 
    \end{lemma} 
    \begin{proof}
        We induct on the degree $d$ of $[k':k]$. If $d = 1$, then $k' = k$ and there is nothing to prove since $b$ is a block idempotent. If $d = p$, then either $b$ remains a block idempotent of $k'G$, in which case this is again vacuously true, or $b$ decomposes into at most $p$ blocks (if $b$ were to decompose into more than $p$ blocks, then since a trace sum over $\Gal(k'/k)$ yields a block of $k$, $b$ would not be primitive). If the latter occurs, then since $b$ decomposes, each new block idempotent in the decomposition must have some coefficients strictly not in $k$, and therefore any element of $\Gal(k'/k)$ acts nontrivially on the set of block idempotents. Since $|\Gal(k'/k)|= p$, the action must be faithful, so $b$ decomposes into exactly $p$ blocks in $k'G$, the set of which is acted transitively and faithfully upon by $\Gal(k'/k).$
        
        Now suppose $d > p$. Let $l$ denote the unique finite field extension of $k$ for which $[k':l] = p$. Then we have a decomposition $b = b_1 + \dots + b_n$ in $lG$, and by the inductive hypothesis, $\Gal(l/k)$ acts transitively on the set $\{b_1, \dots, b_n\}$. In particular, since the Frobenius endomorphism $x \mapsto x^{|k|}$ generates $\Gal(l/k)$, for all integral $1 \leq i, j \leq n$ there exists some integer $x\in \Z$ for which $b_i^x = b_j $. 
        
        Let \[b = c_1^1 + \dots + c_1^{a_1} + c_2^1 + \dots + c_{n-1}^{a_{n-1}}+ c_n^1 +\dots + c_n^{a_n}\] be the decomposition of $b$ into blocks of $k'G$, where $c_i^{i'}b_j \neq 0$ if and only if $i = j$. In particular, for all integers $1 \leq i \leq n$, we have a decomposition $b_i = c_i^1 + \dots + c_i^{a_i}$ (note from the base case $d = p$, $a_i$ can only be 1 or $p$). By induction from considering the extension $k'/l$, there exists some integer $\alpha$ for which $(c_i^{j_1})^\alpha = c_i^{j_2}$ for all $1 \leq i \leq n$ and $1\leq j_1, j_1 \leq a_i$. Moreover, since the Frobenius endomorphism $x \mapsto x^{|k|}$ generates $\Gal(l/k)$, every element of $\Gal(l/k)$ extends to an element of $\Gal(k'/k)$. In particular, if $b_i^x = b_j$ for some integers $1\leq i , j \leq n$, then there exist integers $i', j'$ for which $(c_i^{i'})^x = c_j^{j'}$, and we conclude that $\Gal(k'/k)$ acts transitively on the set of block idempotents $b'$ of $k'G$ for which $b'b \neq 0$. Note this also implies $a_1 = a_2 = \cdots = a_n$. 
    \end{proof}
    
    \begin{corollary}\label{cor:descent}
        Let $k'/k$ be a finite field extension and let $G, H$ be finite groups. Let $b, c$ be block idempotents of $kG$ and $kH$ respectively. There exists a splendid Rickard equivalence $X' \in Ch^b({}_{k'Gb}\triv_{k'Hc})$ which satisfies ${}^\sigma X' \cong X'$ for all $\sigma \in \Gal(k'/k)$ if and only if there exists a splendid Rickard equivalence $X\in Ch^b({}_{kGb}\triv_{kHc})$.
    \end{corollary}
    \begin{proof}
        Let $(K ,\calO, k) \subseteq (K',\calO', k')$ be an extension of $p$-modular systems with $\calO, \calO'$ absolutely unramified (this exists since we may take $\calO = W(k)$ and $\calO' = W(k')$). 
        
        To prove the reverse direction, suppose $X$ exists. By unique lifting of splendid equivalences (\cite[Theorem 5.2]{R96}), there exists a splendid Rickard equivalence $Y$ of $(\calO Gb, \calO Hc)$-bimodules, unique up to isomorphism, which satisfies $k \otimes_\calO Y \cong X$. \cite[Proposition 4.5(a)]{KL18} then asserts that $Y' = \calO' \otimes_\calO Y$ induces a Rickard equivalence between $\calO' Gb$ and $\calO' Hc$ and \cite[Lemmas 5.1 and 5.2]{KL18} assert that $Y'$ is splendid. Then $X' := k' \otimes_{\calO'} Y'$ is a splendid equivalence for $k'Gb$ and $k'Hc$ and is stable under $\Gal(k'/k)$-action by Theorem \ref{thm:galoisstablelift}(b). 
        
        For the forward direction, let $X'$ be as given, and by \cite[Proposition 5.4]{BM23} we may assume $X'$ has no contractible summands. By unique lifting of splendid equivalences, there exists a splendid Rickard equivalence $Y'$ of $(\calO' Gb, \calO' Hc)$-bimodules, unique up to isomorphism, which satisfies $k' \otimes_{\calO'} Y' \cong X'$. Note $k'Gb$ and $k'Hc$ may no longer be indecomposable. Let $b = b_1+\dots+ b_m$ and $c = c_1+\dots +c_n$ denote the the decomposition of $b$ and $c$ into blocks of $k'G$ and $k'H$ respectively (using unique lifting of block idempotents for ease of notation). By \cite[Theorem 8.4]{BM23}, we have that $m = n$, and there exists a permutation $\sigma \in S_n$ such that $Y'$ decomposes into a direct sum of indecomposable splendid Rickard equivalences $Y' = Y'_1 \oplus \dots \oplus Y'_n$ with $Y_i$ a splendid Rickard equivalence between $b_i$ and $c_{\sigma(i)}$ for all $i \in \{1,\dots, n\}$. 

        Since ${}^\sigma Y' \cong Y'$ for all $\sigma \in \Gal(k'/k)$, the action of $\Gal(k'/k)$ partitions the set $\{Y'_1, \dots, Y'_n\}$ into orbits, and Lemma \ref{lem:faithfulact} implies that in fact, the action is transitive. Let $\Omega \leq \Gal(k'/k)$ denote the stabilizer of the action, and let $\calL$ (resp. $l$) denote the absolutely unramified extension $\calO' \supseteq \calL \supseteq \calO$ (resp. extension $k'/l/k$) corresponding to $\Omega$. \cite[Proposition 6.3]{KL18} gives that $(Y'_1)_\calL$ satisfies $\calO' \otimes_\calL (Y'_{1})_\calL \cong Y'_{1} \oplus \dots\oplus  Y'_{n}$, so $Y := (Y'_{1})_\calL$ is a splendid Rickard equivalence between $b \in Z(lGb)$ and $c\in Z(lHc)$. Since $X'$ was assumed to have no contractible summands, it follows that $Y$ has no contractible summands as well. Then \cite[Theorem 8.4]{BM23} implies $Y$ is indecomposable, since it follows that $b$ is a block of $lG$. 
         
        Theorem \ref{thm:galoisstablelift} asserts that there exists a unique chain complex $Y''$ of $(\calO Gb, \calO H c)$-bimodules satisfying $\calL \otimes_\calO Y'' \cong Y$. \cite[Proposotion 4.5(a)]{KL18} asserts that $Y''$ induces a Rickard equivalence between $\calO Gb$ and $\calO Hc$, and \cite[Lemma 5.2]{KL18} asserts that $Y''$ is splendid. Taking $X := k\otimes_\calO Y''$ completes the proof. 
    \end{proof}
    
    Note the previous corollary holds when $k = \F_p[b] = \F_p[c]$. We are now ready to prove Theorem \ref{thm:thmdgen}, our analogue of \cite[Theorem D]{BY22}. This theorem also strengthens \cite[Theorem 6.5]{KL18} by allowing for larger finite field extensions. 
    
    \begin{proof}[Proof of Theorem \ref{thm:thmdgen}]
        Let $k'$ be the field extension $k \supseteq k' \supseteq \F_p$ corresponding to $\stab_\Gamma(b) \leq \Gamma$ via the Galois correspondence. It follows that $k' = \F_p[b] = \F_p[c]$, and that $\Gal(k/k') = \stab_\Gamma(b) = \stab_\Gamma(c) = \stab_\Gamma(X)$. By Corollary \ref{cor:descent}, there exists a splendid Rickard equivalence $\tilde{X}$ for $k'Gb$ and $k'Hc$. The result now follows by \cite[Theorem 6.5(b)]{KL18} after lifting $\tilde{X}$ to $\calO'$, which may be done via \cite[Theorem 5.2]{R96}.
    \end{proof}

    Proving an analogue of \cite[Theorem B]{BY22} follows nearly identically to the $p$-permutation case, and of the groundwork has already been laid out via \cite[Proposition 5.13]{BY22} and \cite[Lemma 5.14]{BY22}. We summarize the setup, and refer the reader to \cite[Section 5]{BY22} for more details and proofs of the following claims.    
    
    \textbf{Setup.} 
    From here, we assume $G$ is a $p$-nilpotent group, so $G$ has a normal $p'$-subgroup $N = O_{p'}(G)$ and $G/N$ is a $p$-group. Let $(K, \calO, k)$ be a $p$-modular system large enough for $G$, and set $\Gamma := \Gal(k/\F_p).$ Fix a block idempotent $b$ of $kG$ and denote by $\tilde{b} := \tr_\Gamma(b)$ the corresponding block idempotent of $\F_p G$. Let $e$ be a block idempotent of $kN$ such that $be \neq 0$. Set $S := \stab_G(e)$, then \[b = \sum_{g\in G/S} {}^ge,\] and $e$ is also a block idempotent of $kS$.  Let $Q \in \Syl_p(S)$. Then $Q$ is a defect group of the block idempotents $e$ of $kS$, $b$ of $kG$, and $\tilde{b}$ of $\F_p G$. Set $\tilde{e} = \tr_\Gamma(e)$ and set $\tilde{S} := \stab_G(\tilde{e})$. Then by \cite[Lemma 5.1]{BY22}, $S\trianglelefteq \tilde{S}$ and \[\tilde{b} = \sum_{G/\tilde{S}} {}^g\tilde{e}.\] 
    
    Set \[e' := \sum_{\tilde{s}\in \tilde{S}/S}{}^{\tilde{s}}e.\] $e'$ is a block idempotent of $k\tilde{S}$. \cite[Lemma 5.3]{BY22} asserts $\tr_\Gamma(e') = \tilde{e}$ and $\tilde{e}$ is a block idempotent of $\F_p \tilde{S}$. Set $H := N_G(Q)$, which is a $p$-nilpotent group, and set $M := O_{p'}(H)$. Then $M = H\cap N = C_N(Q)$. Let $c$ denote the block idempotent of $kH$ which is in Brauer correspondence with $b$. It follows that $\stab_\Gamma(c) = \stab_\Gamma(b)$ and $\tilde{c} := \tr_\Gamma(c)$ is the Brauer correspondent of $\tilde{b}$. 
    
    Let $V$ denote the unique simple $kNe$ module. Since $V$ is absolutely irreducible, it extends to a simple $kSe$ module which we again denote by $V$. Let $f$ denote the block idempotent of $kM$ whose irreducible module is the Glaubermann correspondent of the $Q$-stable irreducible module $V \in {}_{kN}\catmod$. $f$ remains a block idempotent of $kT$, where $T = H \cap S$. It follows that the block idempotents $e$ of $kS$ and $f$ of $kT$ are Brauer correspondents. Let $\tilde{f} := \tr_\Gamma(f), \tilde{T} := \stab_H(\tilde{f}),$ and $f' = \tr^{\tilde{T}}_T(f).$ It follows that \[\stab_\Gamma(f') = \stab_\Gamma(b) = \stab_\Gamma(c) = \stab_\Gamma(e'),\] and $\tilde{T} = H \cap \tilde{S}$.
    
    Now, we act under the assumption that $\Res^S_Q V$ has an endosplit $p$-permutation resolution $X_V$. In fact, $\Res^S_Q V$ is a capped endopermutation module. Note that if $Q$ is abelian, this condition is satisfied, since every indecomposable endopermutation module for an abelian $p$-group is a direct summand of tensor products of inflations of Heller translates of the trivial module of quotient groups, and every indecomposable endopermutation module is absolutely indecomposable. Under this assumption, there exists a direct summand $Y_V$ of $\Ind^S_Q X_V$ such that $Y_V$ is an endosplit $p$-permutation resolution of $V$ as a $kS$-module, and we may choose $Y_V$ to be contractible-free. Set $\Delta_Q S := \{(nq, q): n\in N, q\in Q\}\leq S\times Q$. The induced chain complex $\Ind^{S\times Q}_{\Delta_Q S}Y_V$ is then a splendid Rickard equivalence between $kSe$ and $kQ$, by \cite[Theorem 7.8]{R96}. 
    
    Let $U$ be the simple $kM$-module belonging to $f$. Set $\Delta_Q T := \Delta_Q S \cap (T\times Q)$. The bimodule \[k\tilde{T}f \otimes_{kT} \Ind^{T\times Q}_{\Delta_Q T}U\] induces a splendid Morita equivalence between $k\tilde{T}f'$ and $kQ$. Altogether, the chain complex 
    \[Z := k\tilde{S}e \otimes_{kS} \Ind^{S\times Q}_{\Delta_Q S} Y_V \otimes_{kQ} \big(k\tilde{T}f \otimes_{kT}  \Ind^{T\times Q}_{\Delta_Q T}U\big)^* \] induces a splendid Rickard equivalence between $k\tilde{S}e'$ and $k\tilde{T}f'$. 
    
    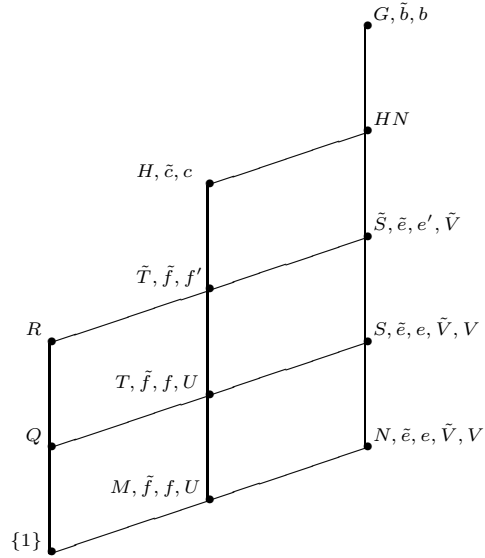
\begin{figure}
        \begin{center}
            \bigskip
            \bigskip
            \bigskip
            \bigskip
            \bigskip
            \unitlength 7mm
            {\scriptsize
                \begin{picture}(7,7)
                    \put(6,9){$\bullet$}  \put(6.2,9.2){$G, \btilde, b$}
                    \put(6,7){$\bullet$}  \put(6.2,7.2){$HN$}
                    \put(6,5){$\bullet$} \put(6.2,5.2){$\Stilde, \etilde, e', \Vtilde$}
                    \put(6,3){$\bullet$} \put(6.2,3.2){$S,\etilde, e, \Vtilde, V$} 
                    \put(6,1){$\bullet$}   \put(6.2,1.2){$N,\etilde,e, \Vtilde, V$}
                    \put(6.05,9.05){\line(0,-1){8}}
                    
                    \put(3,6){$\bullet$} \put(1.7, 6.2){$H, \ctilde, c$}
                    \put(3,4){$\bullet$} \put(1.7, 4.2){$\Ttilde, \ftilde, f'$}
                    \put(3,2){$\bullet$} \put(1.3, 2.2){$T, \ftilde, f, U$}
                    \put(3,0){$\bullet$} \put(1.2, 0.2){$M, \ftilde, f, U$}
                    \put(3.05,6.05){\line(0,-1){6}}
                    
                    \put(6.05,7.05){\line(-3,-1){3}}
                    \put(6.05,5.05){\line(-3,-1){3}}
                    \put(6.05,3.05){\line(-3,-1){3}}
                    \put(6.05,1.05){\line(-3,-1){3}}
                    
                    \put(0,3){$\bullet$} \put(-0.4, 3.2){$R$}
                    \put(0,1){$\bullet$} \put(-0.4, 1.2){$Q$}
                    \put(0,-1){$\bullet$} \put(-0.7, -0.8){$\{1\}$}
                    \put(0.05,3.05){\line(0,-1){4}}
                    
                    \put(3.05,4.05){\line(-3,-1){3}}
                    \put(3.05,2.05){\line(-3,-1){3}}
                    \put(3.05,0.05){\line(-3,-1){3}}
                    
            \end{picture}}
        \end{center}
        \bigskip
        \caption{The lattice of subgroups and blocks for a $p$-nilpotent group $G$.}
        \label{fig:subgroups}
    \end{figure}
    
    Finally, we assume there exists a $W \in {}_{\F_p Q}\catmod$ such that $\Res^S_Q V \cong k \otimes_{\F_p} W$ and that $W$ has an endosplit $p$-permutation resolution $X_W$. Then, the chain complex $k \otimes_{\F_p} X_W$ is an endosplit $p$-permutation resolution of $\Res^S_Q V$ and we may assume $X_V = k \otimes_{\F_p} X_W$. As before, if $Q$ is abelian then this property is satisfied. 
    
    The following theorem is a strengthening of \cite[Corollary 5.15]{BY22}, and the proof follows analogously. 
    \begin{theorem}\label{thm:existenceofSRC}
        Suppose that $R$ is abelian. 
        \begin{enumerate}
            \item There exists a splendid Rickard equivalence between $\F_p\tilde{S}\tilde{e}$ and $\F_p\tilde{T}\tilde{f}$.
            \item There exists a splendid Rickard equivalence between $\F_pG\tilde{b}$ and $\F_pH\tilde{c}$. 
        \end{enumerate}
    \end{theorem}
    \begin{proof}
        \begin{enumerate}
            \item By \cite[Lemma 5.14]{BY22}, we have $\stab_\Gamma(Z) = \stab_\Gamma(e') = \stab_\Gamma(f')$. Hence by Theorem \ref{thm:thmdgen}, there exists a splendid Rickard equivalence $\tilde{Z}$ for $\F_p\tilde{S}\tilde{e}$ and $\F_p\tilde{T}\tilde{f}.$
            \item The $p$-permutation bimodule $\F_pG\tilde{e}$ induces a splendid Morita equivalence, hence a splendid Rickard equivalence, between $\F_pG\tilde{b}$ and $\F_p\tilde{S}\tilde{e}$ see for instance \cite[Theorem 6.8.3]{L182}, by regarding the bimodule as a chain complex concentrated in degree 0. Similarly, the bimodule $\F_pH\tilde{f}$ induces a splendid Rickard equivalence between $\F_pH\tilde{c}$ and $\F_p\tilde{T}\tilde{f}$ in the same way. The result now follows from (a) via the splendid Rickard equivalence \[\F_pG\tilde{e} \otimes_{\F_p\tilde{S}\tilde{e}} \tilde{Z} \otimes_{\F_p\tilde{T}\tilde{f}} \F_pH\tilde{f}\] of $(\F_pG\tilde{b}, \F_pH\tilde{c})$-bimodules.
        \end{enumerate}
        
    \end{proof}
    
    Theorem \ref{thm:thmbgen} now follows immediately from the previous theorem.

    \bibliography{bib}
    \bibliographystyle{plain}
    
\end{document}